%% file: JPRSS.2019.tex
\documentclass[a4paper,12pt]{article}
\usepackage[english]{babel}
\usepackage{amsmath,amsthm,amssymb}
\usepackage{float,color,psfrag}
\usepackage{amstext,amscd}
\usepackage[colorlinks,citecolor=blue]{hyperref}
\usepackage{dsfont}

\usepackage[margin=2cm]{geometry}

\usepackage{tikz}

\usepackage{booktabs}

\usepackage{mathtools}
\mathtoolsset{centercolon}

\usepackage{newpxtext, newpxmath}

\usepackage{enumitem}
\usepackage{xspace}
\usepackage{stackrel}

\newcommand{\Eel}{E_{\mathrm{el}}}
\newcommand{\Ep}{E_{\mathrm{p}}}

\newcommand{\Fcomp}{\mathcal{F}_{\mathrm{comp}}}
\newcommand{\Sflow}{S_{\mathrm{flow}}}
\newcommand{\cflow}{c_{\mathrm{flow}}}

\newcommand{\Ed}{E_{\delta}}

\newcommand{\T}{\mathbb{T}^1}
\newcommand{\spint}{\int_{\T}}

\DeclareMathOperator{\argmin}{argmin}

\renewcommand{\L}{\mathcal{L}}
\renewcommand{\H}{\mathcal{H}}

\newcommand{\1}{\mathds{1}}

\DeclareRobustCommand\circled[1]{\tikz[baseline=(char.base)]{\node[shape=circle,draw,inner sep=1.2pt] (char) {#1};}}

\newtheorem{theorem}{Theorem}
\newtheorem{remark}{Remark}
\newtheorem{lemma}{Lemma}

\title{Modelling adhesion-independent cell migration}

\author{Gaspard Jankowiak\textsuperscript{1}
\and Diane Peurichard\textsuperscript{2}
\and Anne Reversat\textsuperscript{3}
\and Christian Schmeiser\textsuperscript{4}
\and Michael Sixt\textsuperscript{3}}

\date{\today}

\begin{document}

\markboth{Gaspard Jankowiak, Diane Peurichard, Anne Reversat, Chrisitian Schmeiser, Michael Sixt}{Modelling adhesion-independent cell migration}

%%%%%%%%%%%%%%%%%%%%%%%%%%%%%%%%%%%%%%%%%%%%%%%%%%%%%%%%%%%%%%%%%%%%%%%%%%
\maketitle
\begin{center}
\small
1-RICAM-\"Osterreichische Akademie der Wissenschaften, Postgasse 7-9, 1010, Wien, Austria \\
2-INRIA, LJLL, UMPC, 4, place Jussieu, Couloir 16-26, 3e étage 75252 Paris Cedex 05, France\\
3-Institute of Science and Technology Austria, Am Campus 1, 3400 Klosterneuburg, Austria\\
4-Faculty of Mathematics, University of Vienna, Oskar-Morgenstern Platz 1, 1090 Vienna, Austria
\end{center}

\begin{abstract}
A two-dimensional mathematical model for cells migrating without adhesion capabilities is presented and
analyzed. Cells are represented by their cortex, which is modelled as an elastic curve, subject to an internal pressure force. Net polymerization or depolymerization in the cortex is modelled via local addition or removal of material, driving a
cortical flow. The model takes the form of a fully nonlinear degenerate parabolic system. An existence analysis is
carried out by adapting ideas from the theory of gradient flows. Numerical simulations show that these simple rules can account for the behavior observed in experiments, suggesting a possible mechanical mechanism for adhesion-independent motility.
\end{abstract}

{\footnotesize

}

\section{Introduction}
\label{sec:introduction}

One of the most important cellular behaviors is crawling migration. It is observed in many cellular systems both in culture and in vivo \cite{paluch_focal_2016,phillipson_intraluminal_2006}, and involved in many essential physiological or pathological processes (wound healing, embryonic development, cancer metastasis etc. \cite{wolf_molecular_2006}). Since the works of Abercrombie in 1970 \cite{abercrombie_locomotion_1970}, which described the multistep model of lamellipodia-based cell migration, numerous authors have studied the mechanisms of actin-based cell migration. As a result, despite some remaining open questions, lamellipodial migration is now well understood and described\cite{mogilner_mathematics_2009}. This migration mode implies that specific adhesion points transmit intracellular pulling forces from the cytoskeleton to the substrate \cite{rafelski_crawling_2004,vicente-manzanares_integrins_2009}. Actin filaments polymerize below the leading plasma membrane generating pushing forces, and plasma membrane tension resists actin network expansion, pushing back the actin filaments into the cell body. Through adhesion complexes linking the cytoskeleton to the substrate, these retrograde forces are translated into forward locomotion of the cell body \cite{rafelski_crawling_2004}.

Yet, recent studies indicate that cell migration can be achieved without adhesion in confining three-dimensional environments \cite{bergert_force_2015}. Increasing levels of confinement seem to favor adhesion-independent migration in many cell types \cite{friedl_amoeboid_2001}, and can trigger transitions from adhesion-based towards low-adhesive migration modes. However, too strong confinement decreases and even prevents migration \cite{wolf_physical_2013}, due to cell stiffness and nucleus volume. If in the last decade adhesion-independent migration has emerged as a possibly common migration mode, the mechanisms of cell propulsion in this case are still poorly understood. So far in the literature, there is only one known alternative to lamellipodial migration: membrane blebs \cite{blaser_migration_2006,charras_blebs_2008}. These blebs are cellular extensions free of actin filaments, and generated by intracellular hydrostatic pressure \cite{tinevez_role_2009}. Once generated, these blebs grow until a new actin cortex is reassembled inside, which eventually contract and allow the cell to move \cite{yin_lim_computational_2013}. However, numerous studies display cells migrating in an adhesion-free manner without bleb formation.

Several physical mechanisms have been proposed for force transmission between cell and substrate during migration without focal adhesions: (i) cell migration by swimming (by creation of blebs, \cite{yin_lim_computational_2013}), (ii) force transmission based on cell-substrate intercalations of lateral protrusions into gaps in the matrix, (iii) chimneying force transmission where cells push against the obstacles \cite{hawkins_spontaneous_2011,malawista_random_2000} or again (iv) flow-friction driven force transmission. In this last case, mechanisms based on non-specific friction between the cell and the substrate have been investigated to account for adhesion-independent migration \cite{hawkins_spontaneous_2011}. Here, intracellular forces generated by the cytoskeleton are transmitted to the substrate via non-specific friction which has been experimentally measured in \cite{bergert_force_2015,Laemmermann}. The molecular origin of nonspecific friction has not been experimentally investigated. Friction could result from interactions between molecules at the cell surface and the substrate, and unveiling the microscopic origin of nonspecific friction will be an important question for future studies.

In this paper, we propose a simplified 2D model for focal adhesion-independent cell migration, based on the mechanisms (iv). We aim to develop a simplified framework to study whether adhesion-free migration could be primarily driven by simple mechanical features.

%% [Discrete model] We first introduce a discrete model, there the cell is model by its membrane, represented as a set of point masses which are connected through linear springs. The point masses also undergo outwards internal pressure forces. The transfer of material due to polymerization is modelled by adding and removing point masses at specific locations.

Our model is focused on the cell's cortex, which is assumed to be an elastic material confined in the horizontal plane. Because of the cytoplasmic pressure, it is subject to outwards pressure forces. Mathematically, this takes the form of a system of parabolic equations, where the polymerization process leads to an advection-type term. This continuous formulation also allows to properly define the reaction forces compensating mass displacement due to membrane renewal. With this very simple model, we are able to trigger cell migration, with speed depending on the geometrical characteristics of the obstacles. Our results seem to be in qualitative agreement with the biological observations. In Section~\ref{sec:biological observations}, we present the biological observation of leukocyte migrating cells and the experimental setting. Section~\ref{sec:model} is concerned with the mathematical model which is analyzed in Section~\ref{sec:analytical results}. Section~\ref{sec:numerical results} presents our numerical results.

\section{Leukocyte migration in artificial microchannels}
\label{sec:biological observations}

To address the question of adhesion-independent migration of leukocytes, we took advantage of a specific line of lymphocytes that allows for genetic modifications. This enabled us to block the synthesis of talin, an adaptor protein essential for adhesion functionality, using a genetic engineering technique known as CRISPR/Cas9 \cite{Shalem}. In addition, we used well-established microfabricated channels \cite{Vargas} to mimic the confined in vivo environment, coupled to a home-made microfluidic set-up \cite{Reversat}. The top and bottom walls of these channels are flat, and the side walls can have various 
structures. We observe that cells, in which talin has been knocked out, are completely unable to adhere and migrate in channels with flat side walls \cite{Reversat}. Strikingly, their motility is restored in channels with structured side walls
with wave length of the wall structure on the order of magnitude of the cell diameter (see Figure~\ref{FigRatchExp}). This
indicates that adhesion-free motility relies on a structured confinement.

\begin{figure}[H]
\includegraphics[width=\textwidth]{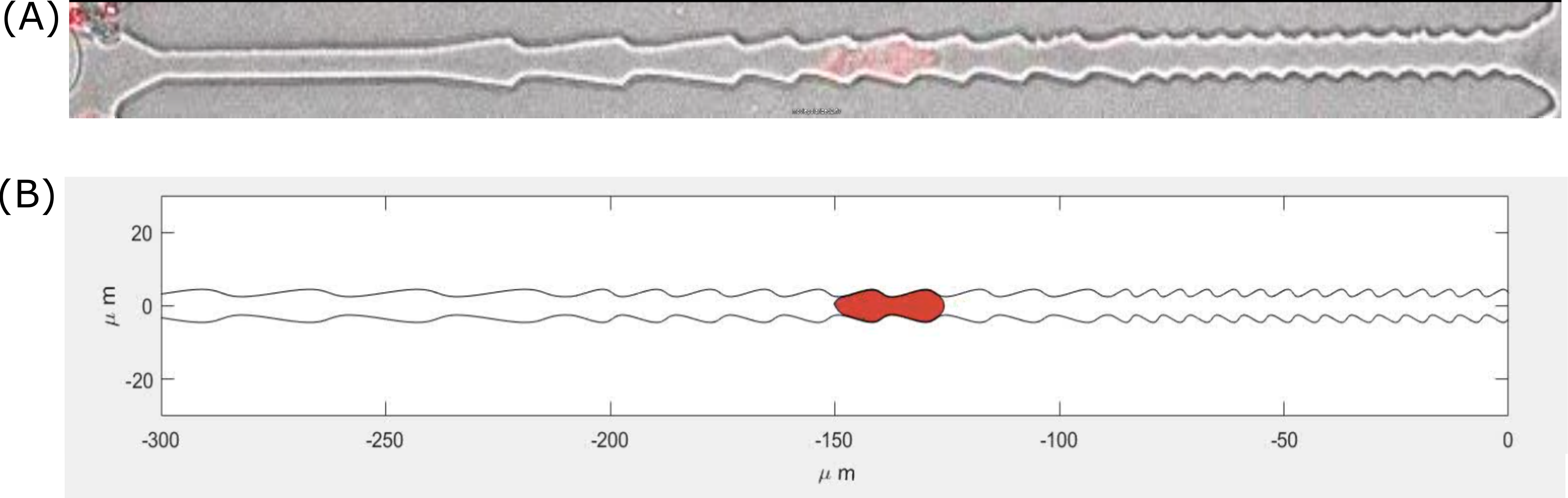}
\caption{(A) Leukocyte (in red) migrating from left to right in a ratchet channel. The channel has three sections with
wavelengths $6\mu m$, $12 \mu m$, and, respectively, $24 \mu m$. (B) Example of a simulation with our mathematical 
model using the setting of the experiment (A). \label{FigRatchExp}}
\end{figure}

\section{The mathematical model}
\label{sec:model}

Since the mechanisms producing the behavior described in the previous section are not known, we propose
a rather simple model for the essential components. The essential idea
is that, guided by a chemotactic signal, the cell polarizes with increased actin polymerization near the front end.
This is assumed to induce a flow of the cell cortex from front to rear, where depolymerization dominates.
The cortex is assumed to be an elastic material with a tendency to equidistribute actin along the cell periphery.
This mechanism, together with a constant cytoplasmic excess pressure (actually the pressure difference
between cytoplasmic and extracellular pressure), determines the cell shape.

The most critical model ingredient are the forces between the cell and its environment. We assume an unspecific friction
with the extracellular liquid, assumed at rest, counteracted by a compensating force, which can be seen
as a consequence of the intracellular transport of actin from rear to front. This compensation is chosen such that
the cell does not move in an unconstrained environment.

\begin{figure}[h!]
\centering
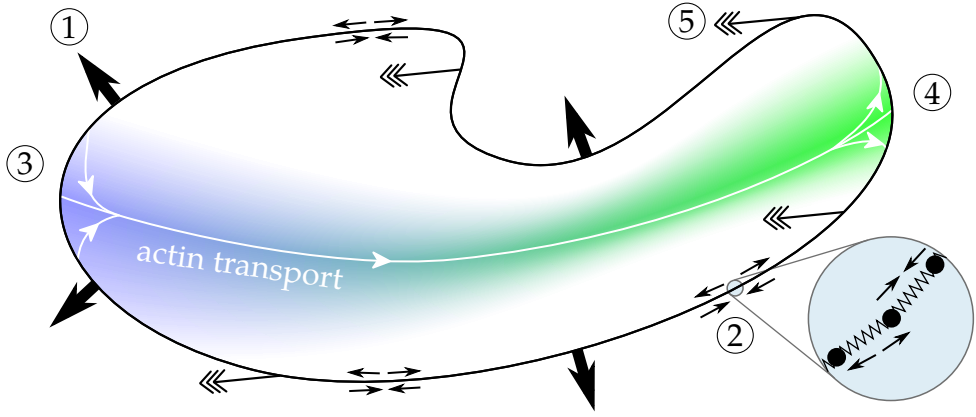
\caption{Schematic description of the cell. The cortex is subject to the a
number of physical effects: \circled{1} pressure forces, \circled{2} linear elasticity
forces, \circled{3}-\circled{4} de-/polymerization, \circled{5} reaction forces
due to transport. Depolymerized actin is transported from the back of the cell
(blue region) to the front (green region) where it becomes part of the cortex
again via polymerization. The transport inside the cell ---in white--- is not
modelled, but by conservation of the center of mass, it results in a reaction
force on the cortex \circled{5}. The inlet show that at the discrete level (for numerical experiments),
on can consider the cortex as a chain of mass points, linked by linear springs.}
%\caption{(A) Schematic view of the cell. Nodes are connected to their two direct neighbors by springs of equilibrium length $\Delta \sigma$. Pressure force acting on membrane points (blue circles) in green, spring force in red. (B) Schematic representation of the process of addition/removal of material. The compensating forces acting on the membrane points to equilibrate the center of mass in a free domain are represented in orange.}
\label{CellForces}
\end{figure}

Motivated by the experimental setup we choose a two-dimensional model, which seems reasonable for
the experimental situation, where the cell is confined between two flat surfaces. Figure~\ref{CellForces} illustrates
a discrete version of the model, where the cell cortex is described by mass points.
We return to this description in Section \ref{sec:numerical results} for simulation purposes, but here we shall formulate
a continuous version of the model, where at time $t\ge 0$ the cortex is represented by a Jordan curve
\begin{equation*}
    \Gamma(t) = \left\{X(s,t) :\, s\in\T\right\} \subset \mathbb{R}^2 \,.
\end{equation*}
The one-dimensional torus $\T$ will be represented by the unit interval, and the variable $s$ corresponds to the amount
of actin material along the cortex, \emph{i.e.} for some non empty interval $[s_1, s_2]$, the amount of actin on the corresponding piece of cortex is $s_2 - s_1$,
this will be formalized later on.
The total amount is normalized to~$1$. The interior of $\Gamma(t)$ is denoted by $\Omega(t)$, such
that $\Gamma(t) = \partial\Omega(t)$. Assuming that $\Gamma(t)$ is smooth enough, we denote by $\tau(t,s)$ and
$n(t,s)$ the unit tangent and unit outward normal vectors. Assuming positive orientation of the parametrization, we have
\begin{equation*}
   \tau = \frac{\partial_s X}{|\partial_s X|} \,,\qquad n = -\tau^\bot \,,
\end{equation*}
with the convention $(a,b)^\bot = (-b,a)$. The notation is illustrated in Figure~\ref{fig:cell parameterization}.

\begin{figure}[h]
    \def\svgwidth{0.5\columnwidth}
    \begin{center}
        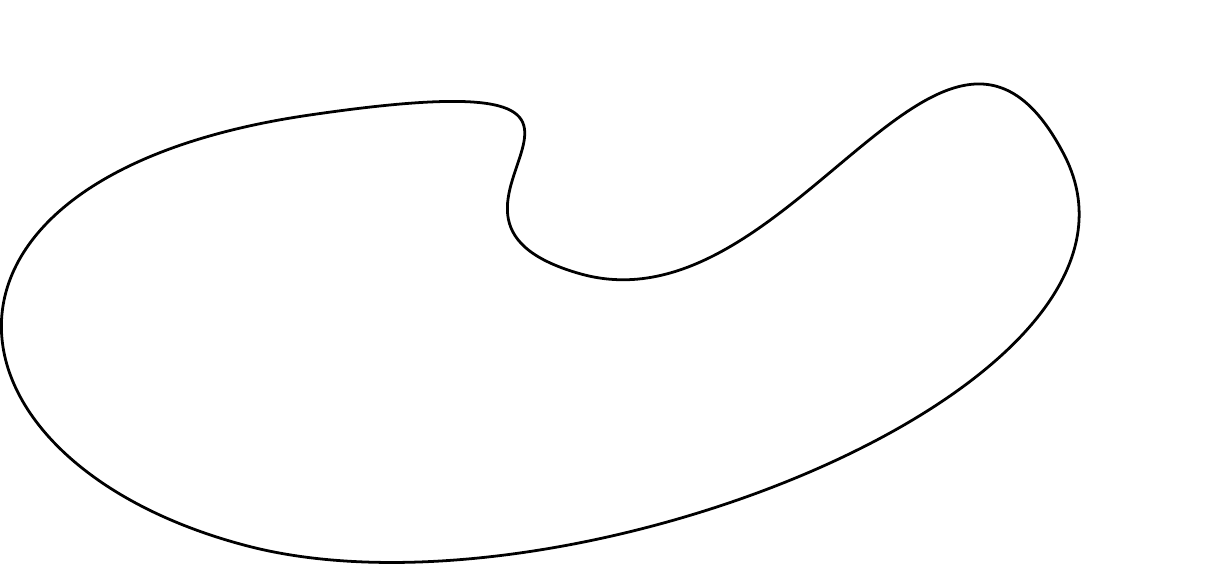
    \end{center}
    \caption{The parameterization and associated vector quantities}
    \label{fig:cell parameterization}
\end{figure}

The cortex is assumed to be elastic and in equilibrium if $|\partial_s X| = 1$, such that~$1$ represents the scaled total
equilibrium length of the cortex. An elastic resistance against stretching, but not against compression, is described
by the potential energy functional
\begin{equation}\label{Espring}
\Eel(X) = \frac{1}{2} \spint \left(|\partial_s X|-1\right)_+^2 ds\,.
\end{equation}
%\noindent where $\kappa_S$ is the stiffness constant linked to the Young's modulus of the membrane.
Neglecting resistance against compression can be seen as a convexification of the elastic energy, which
facilitates the analysis of Section \ref{sec:analytical results}. Actually, we expect the cortex to be always under tension,
such that this assumption should not be relevant from a modelling point of view.

This expectation relies on another model ingredient, a cytoplasmic pressure exceeding the extracellular pressure by
a constant amount $p>0$. The associated potential energy contribution is given by
\begin{equation}\label{Epressure}
    E_\text{p}(X) = -p \left|\Omega(t)\right| = \frac{p}{2} \spint X\cdot \partial_s X^\perp\, ds\,.
\end{equation}
The assumption of a prescribed constant value of $p$ can be seen as a model simplification. The volume of the cell is mainly dictated
by the amount of water it contains, which is subject to osmosis, which can be neglected on the time scales at play here.
It would then be closer to reality to assume a fixed prescribed cell volume, measured by the area $|\Omega(t)|$. In this case $p$
would become a time dependent unknown with the mathematical interpretation of a Lagrange multiplier.

%% Removed the part about area, since we don't use it in the analysis %%
%% Gaspard ~ Fre Okt 13 10:40:58 CEST 2017

% An internal pressure which forces the cell to increase its area $A$ with a maximum allowed area $A_M$. Moreover, we suppose that the cell has a finite compression ability by fixing the minimum cell area to a constant $A_m$. \fxnote{Detail and justify this part of the model.} The pressure energy then reads:
% $$
% \Ep = p \begin{cases}
% A_M - A \qquad \text{ if } A_m<A<A_M\\
% \frac{1}{2} (A_M - A)^2 \qquad \text{ if } A>A_M\\
% \frac{1}{2} (A_m - A)^2 + A_M-A_m\qquad \text{ if } A<A_m,
% \end{cases}
% $$
% \noindent where, using the standard counterclock-wise orientation, the cell area reads:
% $$
% A = -\frac{1}{2} \int_0^L X \cdot \partial_s X^\perp ds.
% $$
% \noindent By this mean, we prevent high compression (case $A<A_m$) and high stretching (case $A>A_M$) of the cell. Indeed in these cases, note that $\Ep$ is a quadratic penalization of the constraint $A=A_m$ or $A=A_M$ respectively, forcing the cell to have an area comprised between $A_m$ and $A_M$. Note that the compressed case is more expensive in terms of energy than the stretched one.

The space constraints, \emph{i.e.} the channel walls, are modelled by requiring
\begin{equation}
\Omega(t)\subset\Omega_c\,,
\label{eq:hard space constraint}
\end{equation}
where $\Omega_c\subset \mathbb{R}^2$ denotes the inside of the channel. For analysis purposes the obstacles
will be softened by introducing the energy contribution
\begin{equation}\label{Eobst}
  E_{\text{obst},\delta}(X) = \spint W_\delta(X) ds \,,\qquad W_\delta = \left( \delta + \rho_\delta * \1_{\Omega_c} \right)^{-1} \,,
\end{equation}
where $\1_{\Omega_c}$ denotes the indicator function of the admissible region, the small parameter $\delta>0$
measures the softness of the obstacle, and $\rho_\delta$ is a positive regularization kernel, approximating the
Delta-distribution as $\delta\to 0$. The formal limit of $E_{\text{obst},\delta}$ as $\delta\to 0$ takes the value $1$ when the
constraint \eqref{eq:hard space constraint} is satisfied, and the value infinity when it is violated.

%The model described above is incomplete, in the sense the dynamics will be that of a relaxation to equilibrium, without actual movement of the cell through the medium. In order to get this movement, we start from the (formal) equations \eqref{eq:main initial condition}-\eqref{eq:differential E} satisfied by $X$ and add polymerization to the model. These equations can also be derived from Newton's Second Law, in an overdamped regime.
%\begin{equation}\label{motionContnopolym}
%\frac{dX}{dt} = \partial_s \big(\kappa_S(|\partial_s X|-1) \frac{\partial_s X}{|\partial_s X|}\big) - f(p) \partial_s X^\perp.
%\end{equation}

If only the three previous energy contributions are considered, we would expect a relaxation to an equilibrium.
Movement requires an active component, coming from actin polymerization and depolymerization in the cortex.
We assume that cell polarization manifests itself by local imbalances of this process producing a net increase of
actin close to the cell front and a decrease close to the rear of the cell. For simplicity we make the equilibrium
assumption that the total amount of actin in the cortex does not change.

We introduce the arclength $l$, which is given by
\begin{equation}\label{s2l}
l(s,t) = \int_0^s |\partial_s X(\sigma,t)|d\sigma \,.
\end{equation}
This relation between the arc length $l$ and the Lagrangian variable $s$ can be inverted in terms of the actin density $\rho(l,t)$ per arc length:
\begin{equation*}
   s(l,t) = \int_0^l \rho(\lambda,t)d\lambda \,,
\end{equation*}
implying $|\partial_s X| = \rho^{-1}$.
We denote by $f(l,t)$ the rate of actin increase ($f>0$) or decrease ($f<0$) per time unit and per arc length, so that $\rho$ satisfies
\begin{equation*}
    \partial_t \rho = f\,.
\end{equation*}
The above mentioned equilibrium assumption translates to
\begin{equation}
   \int_{\Gamma} f\,dl = \spint f(l(s,t))|\partial_s X(s,t)|ds = 0 \,,\qquad t\ge 0\,.
\end{equation}
We then obtain the material derivative for functions of $s$:
\begin{equation}
    \frac{D}{Dt} = \partial_t + \left(\int_0^{l(s,t)} f(\lambda,t)d\lambda\right) \partial_s\,,
\end{equation}
which has to be understood relative to the arc length along $\Gamma$, measured from the point $X(0,t)$. In particular,
the velocity of the cortex relative to the laboratory coordinates is given by $DX/Dt$.

Whereas $f$ is the actin growth density w.r.t. arclength, we need a description in terms of the variable $s$.
\begin{equation}
  g(s,t) = f(l(s,t),t) |\partial_s X(s,t)| \,,
\end{equation}
This roughly gives growth and decay rates per actin filament. This leads to
\begin{equation}
    \frac{D}{Dt} = \partial_t + \left(\int_0^s g(\sigma,t)d\sigma\right) \partial_s\,,
\end{equation}

Friction between the cell surface and the surrounding fluid, which is assumed nonmoving, is modelled by
the $L^2$-gradient flow for the total energy with the contributions \eqref{Espring}, \eqref{Epressure}, and \eqref{Eobst},
where the friction force is given in terms of the material derivative:
\begin{eqnarray*}
  \frac{DX}{Dt} &=& -\Eel'(X) - E_\text{p}'(X) - E_{\text{obst},\delta}'(X) \\
  &=& \partial_s \left((|\partial_s X|-1)_+ \frac{\partial_s X}{|\partial_s X|}\right) - p\,\partial_s X^\bot - \nabla W_\delta(X) \,.
\end{eqnarray*}
Note that the friction coefficient has been eliminated by an appropriate choice of the time scale.
This is however not yet the final form of the model. It would predict movement of freely floating cells in the absence of any
confinement, as can easily be seen by integrating the equation in the absence of the last term with respect to $s$:
$$
  \frac{d}{dt} \spint X\,ds = - \spint \left(\partial_s X \int_0^s g(\sigma,t)d\sigma \right) ds
  = \spint Xg\, ds \,.
$$
Since the right hand side will in general be different from zero, the center of mass will move. An explanation is the force,
used to move depolymerized G-actin from regions with $g<0$ to regions with $g>0$. We do not describe the corresponding mechanism further,
which could include diffusion or transport by molecular motors, for example.
By the action-reaction principle it
seems reasonable to introduce a compensating counterforce with density $\Fcomp(s,t)$, acting on the cortex:
\begin{equation}\label{full-equ}
  \partial_t X  + \partial_s X \int_0^s g\,d\sigma
  = \partial_s \left((|\partial_s X|-1)_+ \frac{\partial_s X}{|\partial_s X|}\right) - p\,\partial_s X^\bot - \nabla W_\delta(X)
  + \Fcomp\,.
\end{equation}
It can be chosen arbitrarily, except that the total force is fixed:
\begin{equation}
  \spint \Fcomp \,ds = - \spint Xg \,ds  \,.
  \label{eq:compensating force conservation}
\end{equation}
We do not make a choice for $\Fcomp$ at this point, the precise choice made for
our numerical experiments will be discussed in Section~\ref{sec:numerical results}.
One can note that $\Fcomp$ is roughly directed from the front towards to back of the cell,
since $g$ is negative (respectively positive) where the depolymerization (respectively polymerization).

\section{Existence results}
\label{sec:analytical results}

\subsection{Formulation of the main results}

We shall prove a global existence result for the initial value problem
for \eqref{full-equ} and a convergence result for the limit $\delta\to 0$ of hard channel walls, \emph{i.e.} strict enforcement of the constraint \eqref{eq:hard space constraint}.

Since the terms describing the cortical flow are chosen to have a vanishing integral
with respect to $s$, we introduce $\Sflow$, which is chosen such that
\begin{equation*}
    \partial_s \Sflow[X] = \partial_s X \int_0^s g \; d\sigma - \Fcomp\,.
\end{equation*}
The initial value problem can then be written in the form
\begin{eqnarray}\label{full-equ1}
  && \partial_t X  + \partial_s (\Sflow[X])
  = \partial_s \left((|\partial_s X|-1)_+ \frac{\partial_s X}{|\partial_s X|}\right) - p\,\partial_s X^\bot - \nabla W_\delta(X) \,,\\
  && X(s,0) = X^0(s) \,. \label{IC}
\end{eqnarray}
For the function spaces, we use the abbreviations $\L := L^2(\T)^2$ and $\H := H^1(\T)^2$ with the norms $\|\cdot\|_2$ and $\|\cdot\|_{1,2}$, respectively.
The torus $\T$ is represented by the $s$-interval $(0,1)$.
For the problem parameters we shall use the following assumptions:

\begin{enumerate}[label=(A\arabic*), ref=(A\arabic*)]
    \item \label{ass:pressure} $p < 2\pi$,
    \item \label{ass:cortex flow} $\|\Sflow[X]\|_2 \le \cflow\|X\|_2$, \quad $\|\partial_s (\Sflow[X])\|_2 \le \cflow\|X\|_{1,2}$.
    \item \label{ass:potential} $W_\delta$ is given by \eqref{Eobst}, where the domain $\Omega_c\subset\mathbb{R}^2$ has a smooth boundary,
and  $x\cdot\nabla W_\delta(x)\ge 0$ for $x\in\mathbb{R}^2$ and for $\delta$ small enough.
    \item \label{ass:initial data} $X^0\in \H$ and $\{X^0(s):\, s\in (0,1)\} \subset\Omega_c$.
    %\item[(A1).] \label{ass:pressure} $p < 2\pi$,
    %\item[(A2).] \label{ass:cortex flow} $\|\Sflow[X]\|_2 \le \cflow\|X\|_2$, \quad $\|\partial_s (\Sflow[X])\|_2 \le \cflow\|X\|_{1,2}$.
    %\item[(A3).] \label{ass:potential} $W_\delta$ is given by \eqref{Eobst}, where the domain $\Omega_c\subset\mathbb{R}^2$ has a smooth boundary,
%and  $x\cdot\nabla W_\delta(x)\ge 0$ for $x\in\mathbb{R}^2$ and for $\delta$ small enough.
    %\item[(A4).] \label{ass:initial data} $X^0\in \H$ and $\{X^0(s):\, s\in (0,1)\} \subset\Omega_c$.
\end{enumerate}

Assumption \ref{ass:pressure} can be motivated by looking at the simplified problem without cortical flow and without obstacles,
i.e. $\nabla W_\delta = \Sflow =0$. In this case the expected circular equilibrium state $X_{equ}$
exists only under Assumption \ref{ass:pressure} and is given by
\begin{equation*}
   X_\text{equ}(s) = \frac{1}{2\pi-p} \bigl(\cos(2\pi s), \sin(2\pi s)\bigr) \,.
\end{equation*}
Assumption \ref{ass:cortex flow} is somewhat restrictive compared to the models discussed in Section \ref{sec:model}. It can be satisfied for bounded actin growth rate $g$. The inequality in Assumption \ref{ass:potential} means roughly
that the allowed domain $\Omega_c$ is star shaped (and that its center has been taken as the origin). It is a technical assumption which in general precludes the type of channels used experimentally.
Finally, Assumption \ref{ass:initial data} on the initial cortex shape implies that the elastic and pressure energies are finite and that the
cell lies in the admissible region. It might be noted that we do not assume that $X_0$ parametrizes a simple curve.
We do not refer to this property since our results do not guarantee that it is preserved globally in time.

\begin{theorem}(Global existence for the penalized problem)
    \label{thm:existence penalized}
    Let the Assumptions \ref{ass:pressure}--\ref{ass:initial data} hold and let $\delta > 0$ be small enough.
    Then there exists a solution $X_\delta$ of \eqref{full-equ1}, \eqref{IC}, such that
    \begin{equation*}
       X_\delta \in H_\mathrm{loc}^1(\mathbb{R}_+; \L) \cap L_\mathrm{loc}^\infty (\mathbb{R}_+; \H) \,,
    \end{equation*}
    uniformly with respect to $\delta\to 0$.
\end{theorem}

\begin{remark}\label{rem1}
    By the Morrey inequality proved in the appendix, $X_\delta$ is H\"older continuous with exponent $\frac{1}{2}$ in terms of $s$
    and $\frac{1}{4}$ in terms of $t$, again uniformly in $\delta$.
\end{remark}

The proof will be carried out in the following three sections. It relies on methods for gradient flows, although they cannot
be applied in a straightforward way. The terms on the right hand side of \eqref{full-equ1} are the $L^2$-gradients of the
energy functionals $\Eel$, $E_\text{p}$, and $E_{\text{obst},\delta}$ (see Section 3), the first of which is convex. The second
and third are treated as continuous perturbations. The cortical flow term on the left hand side is nonvariational.

Our approach is based on a semi-implicit time discretization, where the cortical flow term is evaluated at the old time
step. This allows to solve the discrete problem by energy minimization. A priori estimates for the discrete solution
allow to pass to the continuous limit.

%Some of the first contributions on gradient flows for non-convex functionals can be found in \cite{marino_curves_1989} and \cite{cardinali_class_1997}
%for example, but there some kind of generalized order two differentiability of the functional is required.

These estimates are also uniform in the penalization parameter $\delta$ for the potential,
so we can also carry out the high penalization limit:
\begin{theorem}(Limit of hard channel walls)
    \label{col:existence constrained}
    With the assumptions of Theorem \ref{thm:existence penalized}, the family $\{X_\delta:\,\delta>0\}$ of solutions
    of \eqref{full-equ1}, \eqref{IC} contains a sequence, converging (as $\delta\to 0$) uniformly on bounded time intervals
    to
       \begin{equation*}
       X \in H_\mathrm{loc}^1(\mathbb{R}_+; \L) \cap L_\mathrm{loc}^\infty (\mathbb{R}_+; \H) \,,
    \end{equation*}
    which satisfies $X\in \overline{\Omega_c}$ on $\mathbb{R}_+\times\T$, \eqref{IC}, and
    \begin{equation}\label{equ:interior}
    \partial_t X  + \partial_s (\Sflow[X])
         = \partial_s \left((|\partial_s X|-1)_+ \frac{\partial_s X}{|\partial_s X|}\right) - p\,\partial_s X^\bot \,,
      \end{equation}
      for all $(s,t)$ such that $X(s,t)\in \Omega_c$.
  \end{theorem}

The properties of $X$ stated in the theorem are not a complete formulation of the obstacle problem. Information
on the behavior at the edges of contact regions is missing. Since the equation is degenerately parabolic,
this is not so obvious. Under the additional assumption of convexity of the permissible set $\Omega_c$, one could
expect that $X$ solves the variational inequality
    \begin{equation*}
    \left\langle \partial_t X  + \partial_s (\Sflow[X]) - \partial_s \left((|\partial_s X|-1)_+ \frac{\partial_s X}{|\partial_s X|}\right)
     + p\,\partial_s X^\bot, Y-X\right\rangle_\L \ge 0 \,,
      \end{equation*}
for all $Y\in\H$ such that $Y(s)\in \overline{\Omega_c}$, $s\in\T$.

For the numerical experiments, the obstacle is not modelled by a potential. For any $X(s, t) \in \Gamma(t)$, the total resulting force is rather
restricted to the tangent cone to $\Omega_c$ at $X(s, t)$. In other words, our numerical simulations will be based on the
formulation
\begin{equation}\label{equ:proj}
    \partial_t X(s, t)  = \mathbf{P}_c\left(\partial_s \left((|\partial_s X|-1)_+ \frac{\partial_s X}{|\partial_s X|}\right)
      - p\,\partial_s X^\bot - \partial_s (\Sflow[X]) \right)\,,
\end{equation}
where the potential $W_\delta$ does not appear. The $\mathbf{P}_c$ is the projection on the tangent cone:
\begin{equation}\label{def:proj}
    \mathbf{P}_c\left({F}(s,t)\right) = \begin{cases}
          ({F}(s,t)\cdot \tau_c) \tau_c &
             \text{ for } X\in \partial\Omega_c \,, \, {F}(s,t)\cdot n_c > 0 \,,
             \\
  {F}(s,t) & \text{ otherwise} \,,
    \end{cases}
\end{equation}
where $\tau_c$ and $n_c$ are a normalized tangent vector and the normalized outward normal
along $\partial\Omega_c$.

\subsection{The energy functional and its properties}

We introduce the total energy functional
\begin{eqnarray*}
    \Ed(X) &:=& \Eel(X) + \Ep(X) + E_{\text{obst},\delta}(X) \\
     &=&  \begin{cases} \spint \left(\frac{1}{2}\left(|\partial_s X| - 1\right)_+^2
 + \frac{p}{2} X\cdot \partial_s X^\perp +  W_\delta(X)\right)ds \quad& \text{if } X\in \H \\
        +\infty & \text{otherwise,}
              \end{cases}
\end{eqnarray*}
and prove its coercivity:

\begin{lemma}[Coercivity]\label{lem:coercivity}
    For $p < 2\pi$ and for all $X \in \L$
    \begin{equation}
        \label{eq:energy lower bound}
        \Ed(X) \ge \frac{1}{2}\left(1-\frac{p}{2\pi}\right) \|\partial_s X\|_2^2 - \|\partial_s X\|_2 \ge -\frac{\pi}{2\pi-p}\,.
    \end{equation}
\end{lemma}

\begin{proof}
W.l.o.g. we assume $X\in \H$ and denote by $\Omega$ the enclosed domain and by $L$
the length of its boundary. Then the isoperimetric and the Cauchy-Schwarz inequalities imply
    \begin{align*}
        \Ep(X) &= -p |\Omega| \ge - p \frac{L^2}{4\pi} = - \frac{p}{4\pi} \Vert \partial_s X \Vert_1^2
          \ge -\frac{p}{4\pi} \Vert \partial_s X \Vert_2^2\,.
    \end{align*}
For the elastic energy again the Cauchy-Schwarz inequality is used:
    \begin{align*}
        \Eel(X) &= \frac{1}{2} \spint (|\partial_s X|-1)^2ds - \frac{1}{2} \int_{|\partial_s X|<1} (|\partial_s X|-1)^2 ds
        \\
                &\geq \frac{1}{2} \left(\|\partial_s X\|^2_2 - 2 \|\partial_s X\|_1 + 1\right) - \frac{1}{2}
                \\
                &\geq \frac{1}{2} \|\partial_s X\|^2_2 - \|\partial_s X\|_2\,,
    \end{align*}
    which completes the proof, since $E_{\text{obst},\delta}\ge 0$.
\end{proof}

\begin{remark}
This result shows that the definition of $\Ed$ has been appropriate, since finiteness of $\Ed(X)$ implies $X \in \H$.
\end{remark}

\begin{lemma}
With respect to the weak topology in $\H$, the functional $\Eel$ is convex and lower semicontinuous,
and the functionals $\Ep$ and $E_{\text{obst},\delta}$ are continuous.
\end{lemma}

\begin{proof}
    The integrand of $\Eel$ is positive, and a convex function of $\partial_s X$. Lower semicontinuity follows from Giaquinta\cite{giaquinta_multiple_1983}, Theorem 2.5. The other properties are
straightforward.
\end{proof}

\subsection{Time discretization}

Choose $\tau>0$ and $X^{n-1}\in\H$. Then, by the results of the preceding section, the functional
\begin{equation*}
   \Phi(\tau,X^{n-1};Y) := \frac{\|Y-X^{n-1}\|_2^2}{2\tau} + \Ed(Y) + \langle Y, \partial_s(\Sflow[X^{n-1}]) \rangle_\L
\end{equation*}
is weakly lower semicontinuous on $\H$. It is also bounded from below since $\Ed$ is, and since
\begin{eqnarray*}
  && \frac{\|Y-X^{n-1}\|_2^2}{2\tau} + \langle Y, \partial_s(\Sflow(X^{n-1}) \rangle_\L
  \\
  &\ge& -\frac{\tau}{2}\| \partial_s(\Sflow[X^{n-1}])\|_2^2 - \|X^{n-1}\|_2 \| \partial_s(\Sflow[X^{n-1}])\|_2 \\
  &\ge&  -\left( \frac{\tau}{2}\cflow^2 + \cflow\right)\|X^{n-1}\|_{1,2}^2 \,.
\end{eqnarray*}
Furthermore, sublevel sets are bounded in $\H$.
This is sufficient for the minimum of $\Phi(\tau,X^{n-1};\cdot)$ to be assumed in $\H$, and we choose
\begin{equation*}
  X^n \in \argmin_{Y\in\H} \Phi(\tau,X^{n-1};Y) \,,\qquad n\ge 1\,.
\end{equation*}
We define $X_\tau\in C^{0,1}_\mathrm{loc}(\mathbb{R}_+,\H)$ as the piecewise linear interpolation of the $X^n$, $n\ge 0$. More precisely, we have
\begin{equation}
    X_\tau(t) := X^{k_t} + \frac{(t - k_t\tau)}{\tau} (X^{k_t+1} - X^{k_t})\,,
    \label{eq:definition x tau}
\end{equation}
where $k_t \in \mathbb{N}$ is such that $k_t \tau \le t < (k_t+1)\tau$.

\begin{lemma}[Uniform estimates for the interpolation]
    Let Assumptions \ref{ass:pressure}--\ref{ass:initial data} hold. Then
\begin{equation*}
   X_\tau \in H_\mathrm{loc}^1(\mathbb{R}_+,\L) \cap L_\mathrm{loc}^\infty(\mathbb{R}_+, \H) \,,
\end{equation*}
uniformly with respect to small enough $\tau$ and $\delta$.
\label{lem:interpolation uniform bounds}
\end{lemma}

\begin{proof}
Denoting the duality bracket with $\langle \cdot, \cdot\rangle$, the variation of $\Phi$ in the direction $Y\in \H$ gives
\begin{equation}\label{discr-var}
  \frac{\langle X^n-X^{n-1}, Y\rangle_\L}{\tau} + \langle \partial \Ed(X^n), Y\rangle
  - \langle\partial_s Y , \Sflow[X^{n-1}]\rangle_\L = 0 \,,
\end{equation}
with the formal gradient of the energy determined from
\begin{eqnarray*}
  \langle \partial \Ed(X^n), Y\rangle_\L &=&
  \spint \left( (|\partial_s X^n| - 1)_+ \frac{\partial_s X^n\cdot \partial_s Y}{|\partial_s X^n|}
  + Y\cdot\bigl(p\partial_s (X^n)^\bot + \nabla W_\delta(X^n)\bigr)\right)ds
\end{eqnarray*}
With $Y=X^n$ and with Assumptions \ref{ass:pressure} and \ref{ass:potential} we get, similarly to the proof of Lemma \ref{lem:coercivity},
\begin{align*}
    \langle \partial \Ed(X^n), X^n\rangle_\L
    &\ge \spint (|\partial_s X^n| - 1)_+ |\partial_s X^n|ds - 2p|\Omega^n|
    \\
    &\ge \left( 1- \frac{p}{2\pi}\right)\|\partial_s X^n\|_2^2 - \|\partial_s X^n\|_2 - \frac{1}{4} \,,
\end{align*}
and we observe
\begin{equation*}
 \langle X^n-X^{n-1}, X^n\rangle_\L  \ge \frac{1}{2}\left(\|X^n\|_2^2 - \|X^{n-1}\|_2^2\right) \,.
\end{equation*}
Finally, we use
\begin{equation*}
 | \langle\partial_s X^n , \Sflow[X^{n-1}]\rangle_\L| \le \gamma \|\partial_s X^n\|_2^2
 + \frac{\cflow^2}{4\gamma}\|X^{n-1}\|_2^2
\end{equation*}
with $0<\gamma<1-p/(2\pi)$. This implies the existence of positive constants $A_1,A_2,A_3$ (independent from $n$, $\tau$
and $\delta$) such that
\begin{equation*}
  \|X^n\|_2^2 + \tau A_1 \|\partial_s X^n\|_2^2 \le (1+\tau A_2)\|X^{n-1}\|_2^2 + \tau A_3 \,.
\end{equation*}
A discrete Gronwall estimate now gives
\begin{equation*}
  \|X^n\|_2^2 \le \|X^0\|_2^2 \,e^{A_2 n\tau} + \frac{A_3}{A_2}\left(e^{A_2 n\tau} - 1\right) \,,
\end{equation*}
and, as a consequence,
\begin{equation}\label{discr-H1}
   \tau \sum_{k=1}^n \|\partial_s X^k\|_2^2 \le C(n\tau) \,.
\end{equation}
From the last two bounds we get for any $t_1, t_2 > 0$
\begin{equation*}
    \sup_{t \in [t_1, t_2]} \|X_\tau(t)\|^2 \le
  \|X^0\|_2^2 \,e^{A_2 k_2\tau} + \frac{A_3}{A_2}\left(e^{A_2 k_2\tau} - 1\right) \,,
\end{equation*}
and
\begin{equation*}
    \| \partial_s X_\tau \|_{L^2([t_1, t_2])}^2 \le \sum_{k=k_1}^{k_2}
    \tau \left(\| \partial_s X^k\|_2^2 + \|\partial_s X^{k+1}\|_2^2\right)
    \le 2 C(k_2 \tau)\,,
\end{equation*}
where $k_2 \tau \le t_2 < (k_2 + 1 )\tau$.
In other words, we have shown
$X_\tau \in L_\mathrm{loc}^\infty(\mathbb{R}_+,\L) \cap L_\mathrm{loc}^2(\mathbb{R}_+,\H)$ uniformly in $\tau$ and $\delta$.

From the minimization we get
\begin{multline*}
  \frac{\|X^n-X^{n-1}\|_2^2}{2\tau} + \Ed(X^n) + \langle X^n, \partial_s(\Sflow[X^{n-1}])\rangle_\L
  \\
  \le \Ed(X^{n-1}) + \langle X^{n-1}, \partial_s(\Sflow[X^{n-1}])\rangle_\L \,.
\end{multline*}
With the velocity $\partial_t X^n_\tau = \frac{X^n-X^{n-1}}{\tau}$ of the linear interpolant this reads
\begin{equation*}
  \frac{\tau}{2} \|\partial_t X^n_\tau\|_2^2 + \Ed(X^n)
  \le \Ed(X^{n-1}) - \tau\langle \partial_t X_\tau^n, \partial_s(\Sflow[X^{n-1}])\rangle_\L
\end{equation*}
With the Young inequality and Assumption \ref{ass:cortex flow} we obtain
\begin{equation*}
  \frac{\tau}{4} \|\partial_t X^n_\tau\|_2^2 + \Ed(X^n)
  \le \Ed(X^{n-1}) + \tau \cflow^2 \|\partial_s X^{n-1}\|_2^2
\end{equation*}
Using \eqref{discr-H1}, summation over $n$ and an application of Lemma \ref{lem:coercivity}
completes the proof.
\\
\end{proof}

\subsection{The continuous and high penalization limits}

Considering the time discrete solution $(X^n)_{n\ge 0}$, we recall the definition \eqref{eq:definition x tau} of $X_\tau$ and define two additional piecewise constant continuous-time approximations:
\begin{equation*}
    \begin{cases}
        %X_\tau(t,s) &:= (n-t/\tau)X^{n-1}(s) + (1-n+t/\tau)X^n(s) \,,\\
        X_\tau^\mathrm{old}(t,s) &:= X^{k_t}(s) \,,
        \\
        X_\tau^\mathrm{new}(t,s) &:= X^{k_t+1}(s) \,,
   \end{cases}
   \quad \text{ where } k_t \tau < t < (k_t+1)\tau \,,
\end{equation*}
Then \eqref{discr-var} implies
\begin{equation}\label{equ:tau}
  \partial_t X_\tau + \partial\Eel(X_\tau^\mathrm{new}) + p\partial_s (X_\tau^\mathrm{new})^\bot + \nabla W_\delta(X_\tau^\mathrm{new})
  + \partial_s \Sflow[X_\tau^\mathrm{old}] = 0 \,,
\end{equation}
where $\partial\Eel(X)$ is the subdifferential of the elastic energy, given by the distributional derivative with respect
to $s$ of $-(|\partial_s X|-1)_+ \frac{\partial_s X}{|\partial_s X|}$.
The equality only holds in the dual space of $\H$ \emph{a priori}, but thanks to Lemma~\ref{lem:interpolation uniform bounds}, it also holds in $\L$.
It is our goal to pass to the limit $\tau\to 0$ in this
equation. By the results of the previous section, there exists a sequence $\tau_k\to 0$ such that
\begin{equation*}
  \lim_{k\to\infty} X_{\tau_k} = \lim_{k\to\infty} X_{\tau_k}^\mathrm{new} = \lim_{k\to\infty} X_{\tau_k}^\mathrm{old} = X
  \qquad\mbox{in } L_\mathrm{loc}^2(\mathbb{R}_+,\L) \,.
\end{equation*}
We also have that all the terms in \eqref{equ:tau} except $\partial\Eel(X_\tau^\mathrm{new})$ are bounded in
$L_\mathrm{loc}^2(\mathbb{R}_+,\L)$ uniformly in $\tau$ and $\delta$. In all these terms we can pass to the limit in the
sense of distributions, but also weakly in $L_\mathrm{loc}^2(\mathbb{R}_+,\L)$.

This also implies weak convergence of $\partial\Eel(X_{\tau_k}^\mathrm{new})$ to some $\eta\in L_\mathrm{loc}^2(\mathbb{R}_+,\L)$,
and we obtain
\begin{equation*}
  \partial_t X + \eta + p \partial_s X^\bot + \nabla W_\delta(X) + \partial_s \Sflow[X] = 0 \,.
\end{equation*}
It remains to identify $\eta$.

% References:
% Brézis, Opérateurs Maximaux Monotones
% exemple 2.3.4 p.25 & § II.7

Since the subdifferential is a maximal monotone operator, we shall apply the Minty trick
\cite[Theorem 2.2]{hungerbuhler_young_2000,minty_monotonicity_1963}.
Testing \eqref{equ:tau} against $X_\tau^\mathrm{new}$, its strong convergence immediately implies
\begin{equation*}
  \lim_{k\to\infty} \int_0^T \langle \partial\Eel(X_{\tau_k}^\mathrm{new}), X_{\tau_k}^\mathrm{new}\rangle_\L dt
    = \int_0^T \langle \eta,X\rangle_\L dt \,,
\end{equation*}
which is sufficient for $\eta\in \partial\Eel(X)$ and, thus,
\begin{equation*}
  \eta = -\partial_s\left((|\partial_s X|-1)_+ \frac{\partial_s X}{|\partial_s X|}\right) \,,
\end{equation*}
to be understood as the weak derivative of an $L^2$-function, since $X\in L_\mathrm{loc}^\infty(\mathbb{R}_+, \H)$.
This completes the proof of Theorem \ref{thm:existence penalized}.

For the proof of Theorem \ref{col:existence constrained}, we note that by the results of Theorem
\ref{thm:existence penalized} and following Remark \ref{rem1} the uniform convergence of a subsequence
to $X$ is immediate. The uniform boundedness of $E_{\text{obst},\delta}$ implies $X\in\overline{\Omega_c}$ by its
continuity and by the uniform convergence, since $W_\delta(x)\to\infty$ for $x\notin \overline{\Omega_c}$.

Similarly, since $\nabla W_\delta(x)\to 0$ for $x\in\Omega_c$, the weak formulation of \eqref{equ:interior} can be derived
by using test functions vanishing away from $\{(s,t):\, X(s,t) \in \Omega_c\}$.
The convergence of the various terms is handled exactly as for the continuous limit.

\section{Numerical results}
\label{sec:numerical results}

\subsection{Discretization}

We start with the situation without obstacles and
introduce $N\in\mathbb{N}$ grid-points for the discretization of $s$ such that $s_i = i\Delta s$,
$\Delta s =\frac{1}{N}$ for $i$ considered on a discrete torus, meaning that $i$ is identified with $i+N$.
The time step is $\Delta t>0$ and $t^n = n\Delta t$, $n\ge 0$. The numerical approximation for $X(s_i,t^n)$ is
denoted by $X_i^n$.

We assume a cell which is polarized in the fixed direction $\omega\in\mathbb{R}^2$, $|\omega|=1$ and define
$i_0^n$ and $i_1^n$ such that
\begin{equation}\label{omega}
  \omega\cdot X_{i_0^n}^n = \max_i \omega\cdot X_i^n \,,\qquad
  \omega\cdot X_{i_1^n}^n = \min_i \omega\cdot X_i^n \,.
\end{equation}
Actin is added to the cortex at the leading end $X_{i_0^n}^n$ and removed at the trailing end $X_{i_1^n}^n$, corresponding
to $s_0(t^n) = i_0^n\Delta s$, $s_1(t^n) = i_1^n\Delta s$ in the notation of Section \ref{sec:model}.
%\begin{agjnote}{More details about the polymerization in the numerics}
%What does it means in terms of computations? If we renumber $X_i$, we need to specify how exactly, and when this is done (at each iteration or not). This is probably also a good place to put/recall the modelling discussion below \eqref{chooseF}.
%\end{agjnote}

We use the explicit Euler scheme for the time discretization and symmetric finite differences for the discretization
in the $s$-direction:
\begin{multline*}
  \frac{X_i^{n+1} - X_i^n}{\Delta t} = {\bf F}[X^n]_i :=
  \\
  v\, \1_{i_1\leq i \leq i_0} \frac{X^n_{i+1}-X^n_{i-1}}{2\Delta s} +
  \frac{G^n_{i+1/2} - G^n_{i-1/2}}{\Delta s} - p \frac{({X^n_{i+1}} - {X^n_{i-1}})^\perp}{2\Delta s}
  + F_{\mathrm{comp},i}^n\,,
\end{multline*}

where the elasticity forces are given by
\begin{equation}\label{Gi}
G_{i+1/2} = \left(\frac{|X_{i+1} - X_{i}|}{\Delta s} - 1 \right)_+\frac{X_{i+1} - X_i}{|X_{i+1} - X_i|} \,,
\end{equation}
and the compensating force by
$$
F_{\mathrm{comp},i}^n = \frac{v}{2} \left( \frac{X_{i_1^n-1}^n + X_{i_1^n}^n}{2} - \frac{X_{i_0^n+1}^n + X_{i_0^n}^n}{2}\right) \left(\tilde\delta_a(i-i_0^n) + \tilde\delta_a(i-i_1^n)\right) \,.
$$
Here $\tilde\delta_a(i)$ is a mollification of the Dirac Delta in $s_i$, with smoothing parameter $a$ such that $\sum_i \tilde\delta_a(i)\Delta s=1$.
The approximation for the compensating force has been chosen such that the discrete center of mass
$\sum_i X_i^n \Delta s$ is independent of time (in the absence of obstacles).

%\subsection{Obstacles}

The restriction to the admissible domain $\Omega_c$ is enforced by an approximation of the formulation \eqref{equ:proj},
\eqref{def:proj}:
$$
  \frac{X_i^{n+1} - X_i^n}{\Delta t} = {\bf P}_{c,\epsilon}({\bf F}[X^n]_i ) \,,
$$
where
\begin{equation*}
  \mathbf{P}_{c,\epsilon}({\bf F}[X]_i) = \left\{ \begin{array}{ll} ({\bf F}[X]_i\cdot \vec{\bf t}) \vec{\bf t} &
             \mbox{for } {\rm dist}(X_i, \partial\Omega_c)<\epsilon \,, \, {\bf F}[X]_i\cdot\vec{\bf n} > 0 \,, \\
                          {\bf F}[X]_i & \mbox{else} \,, \end{array}\right.
\end{equation*}
with the tangent and normal vectors evaluated at the orthogonal projection of $X_i$ to $\partial\Omega_c$
(see Figure~\ref{SketchBoundary}). The width $\epsilon$ of the tube, where the projection is applied, is kept sufficiently large with respect to the time step such that the constraint \eqref{eq:hard space constraint} is met.

\begin{figure}[H]
\centering
\includegraphics[width=0.5\textwidth]{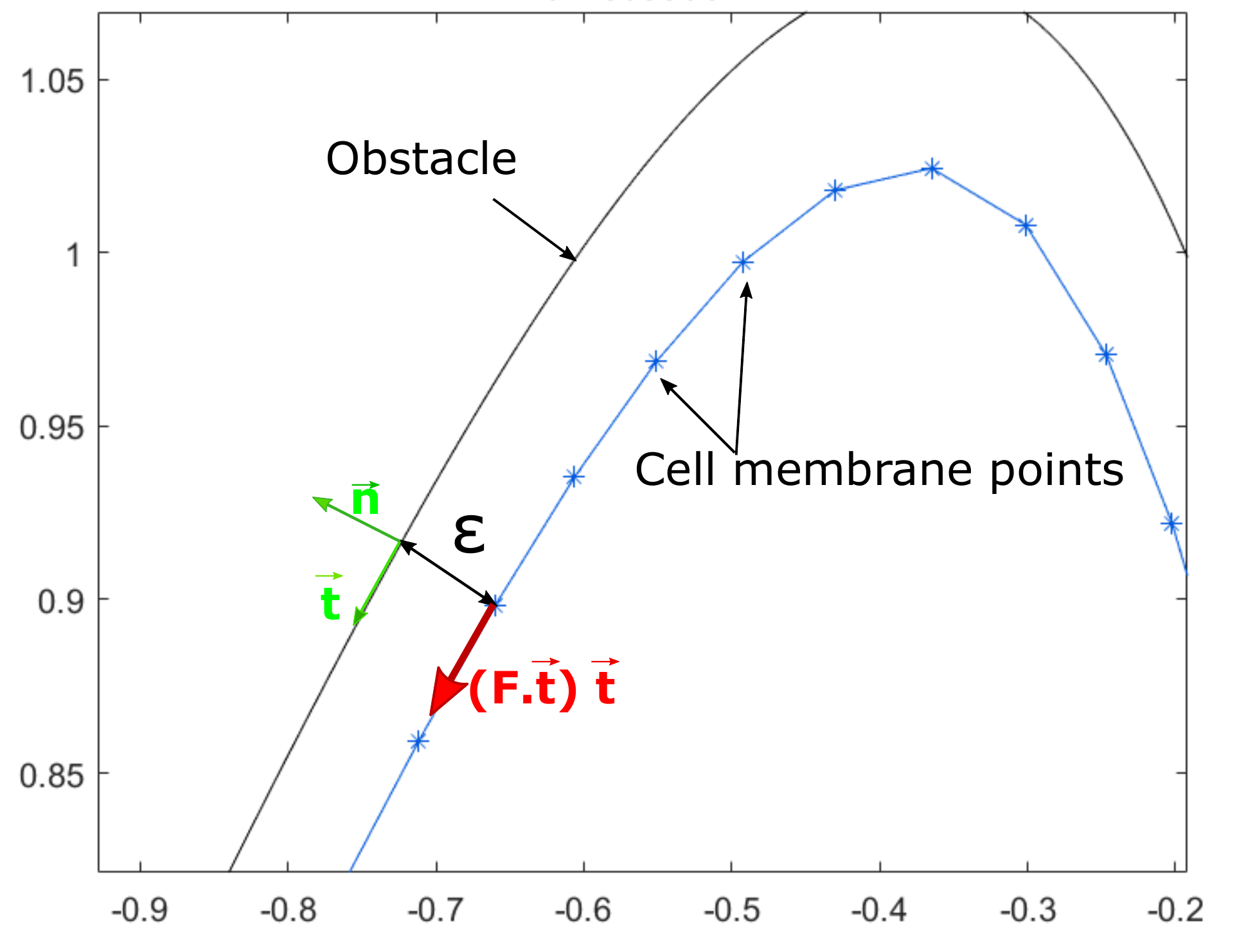}
\caption{Sketch of the projection of the forces acting on the cell cortex close to the walls.}
\label{SketchBoundary}
\end{figure}

% We now aim to perform numerical simulations of the continuous model \eqref{motionCont}, taking in consideration the presence of different types of obstacles. We first study the stationary solutions in a free domain, i.e without any obstacles. If not otherwise stated, the values of the parameters used for the simulations are given in table \ref{table1}.

\subsection{Dimensionalisation of the model parameters}

We denote by $\kappa_S$ the elasticity force and by $\mu$ the internal friction coefficient in front of the time derivative. Note that all these quantities have been set to 1 so far in the model, without loss of generality. We recover them here for the sake of computing their dimensions, Eq. \eqref{equ:proj} reads:
 \begin{equation}
     \label{eq:evolution projection all parameters}
   \mu \partial_t X  = \mathbf{P}_c\left(\kappa_S \partial_s \left((|\partial_s X|-1)_+ \frac{\partial_s X}{|\partial_s X|}\right)
      - p\,\partial_s X^\bot - \frac{1}{ \mu}\partial_s (\Sflow[X]) \right).
\end{equation}
We denote by $x_0$ and $t_0$ the space and time units of the model. The space unit is chosen to be half the cortex length of leucocytes $L = 2 x_0$, i.e $x_0 \approx 47.6 \mu m$. The polymerization speed in leukocytes is around $12 \mu m.min^{-1}$ and the dimensionless polymerization speed writes $v = 2 \frac{x_0}{t_0}$ therefore 1 time unit of the model corresponds to $t_0 \approx 8 min$. The elastic properties of the human red blood cell have been studied via micropipette aspiration experiments \cite{waugh_thermoelasticity_1979} where the authors show that the stiffness constant of leukocytes membrane is of order $k_S = 7 pN.\mu m^{-1}$. In our model $\frac{k_S}{\mu} = \frac{1}{t_0} = 0.1265 min^{-1}$ therefore we can deduce that the internal friction coefficient $\mu \approx 55 pN.\mu m^{-1}.min$. Finally, we choose the pressure constant $p = 3.2$. We deduce $\frac{p}{\mu} = 3.2 \frac{1}{t_0} = 0.4 min^{-1}$ and therefore $p \approx 22 pN.\mu m^{-1}$ which is in biological range. All the model parameters are summarized in Table~\ref{table1}.

\begin{table}[H]
    \begin{center}
\begin{tabular}{cccc}
    \toprule
    \multicolumn{4}{c}{Parameters}
    \\ \midrule
\hline
\multicolumn{4}{c}{Numerical parameters}\\
\hline
Symbol & \multicolumn{2}{c}{Numerical Value} & Description \\
\hline
$\Delta t$ & \multicolumn{2}{c}{$4. 10^{-2}$} & Time step \\
$\Delta s$ & \multicolumn{2}{c}{$5. 10^{-3}$} & Discretization step \\
$N$ & \multicolumn{2}{c}{200 } & Number of nodes\\
$\epsilon$ & \multicolumn{2}{c}{0.1} & Distance to obstacle for projection\\
$a$ & \multicolumn{2}{c}{adapted} & Length of the polymerization zone\\
\hline
\multicolumn{4}{c}{Model parameters}\\
\hline
Symbol & Numerical Value & Biological value & Description \\
\hline
$L$ & 1 & $48 \mu m$& Reference membrane length\\
$L_{\text{eff}}$ & 2 & 94 $\mu m$& Effective membrane length \\
$\kappa_S$ & 1 & 7 $pN.\mu m^{-1}$ & Stiffness constant \\
$p$ & $3.2$ & 22 $pN.\mu m^{-1}$& Pressure force \\
 $v$ & $2$ & 12 $\mu m.min^{-1}$& Polymerization speed \\
$\mu$ & 1 & 55 $pN.\mu m^{-1}.min$ &Internal friction coefficient \\ \bottomrule
 \end{tabular}
 \end{center}
 \caption{Numerical and model parameters for the simulations of the paper. \label{table1}}
\end{table}

\subsection{Simulations without Obstacle}
In Figure~\ref{NoObst} we show computed equilibrium shapes of the cortex in the absence of obstacles. The polarization
of the cell is always to the left (i.e. $\omega = (-1,0)$ in \eqref{omega}).

Figure~\ref{NoObst} (A) shows a simulation without compensating force. As expected, the cell migrates in the direction
of the leading end with a speed slightly less than the polymerization/depolymerization speed. The equilibrium shape
is a circle.

Figures~\ref{NoObst} (B,C,D) show the equilibrium shapes obtained for different values of $h = \frac{a}{L}$,
where we recall that $a$ measures the spread of the compensating force and $L$ is the equilibrium circumference.
The most important observation is that migration has been turned off successfully by including the compensating
force. The deviation of the equilibrium shape from a circle is stronger for more concentrated compensating forces (smaller values of $a$ therefore $h$, compare Figures (D) to (B)).

\begin{figure}[H]
\centering
\includegraphics[width=\textwidth]{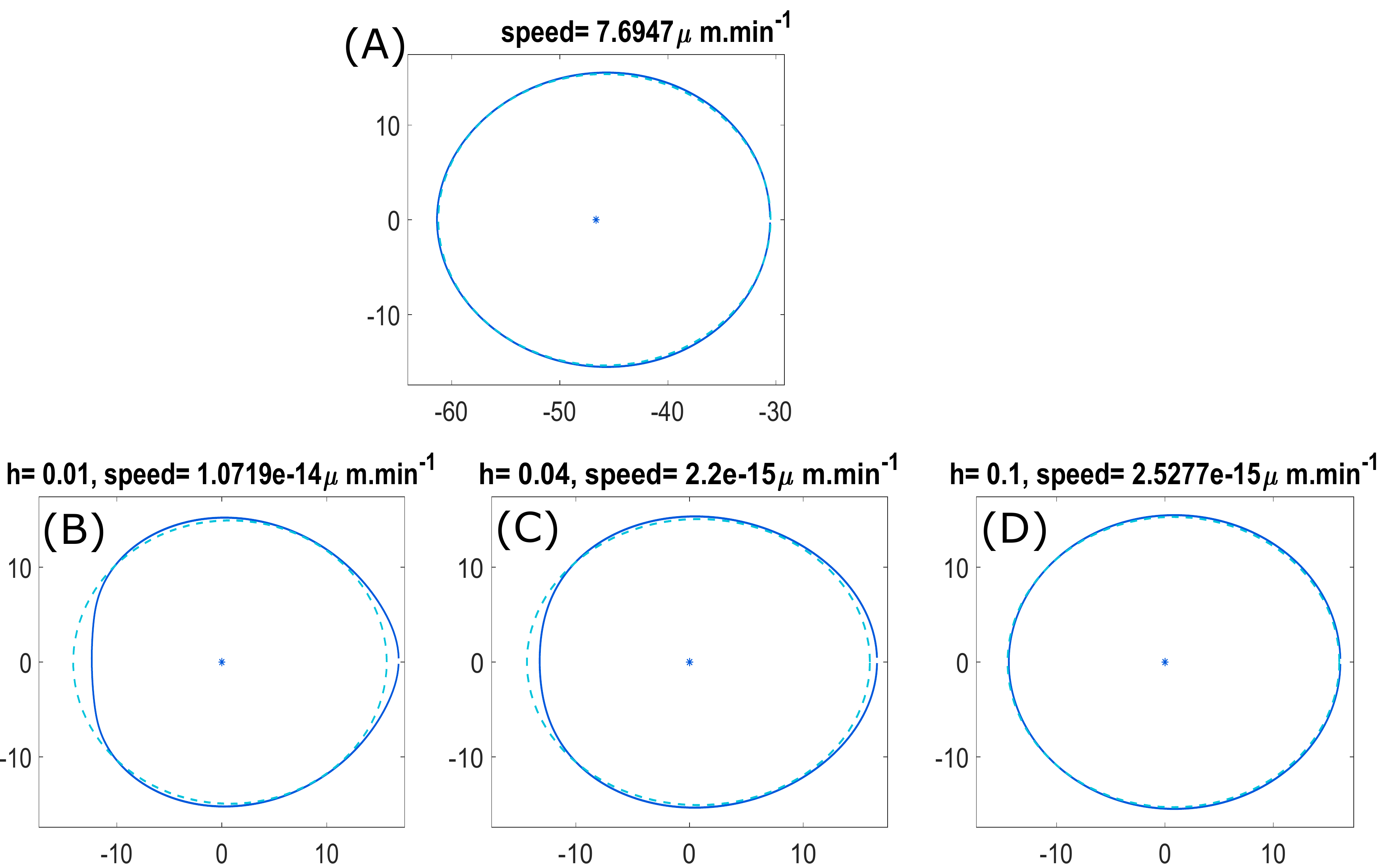}
\caption{Different cell shapes obtained for different values of $h$: (A) Without compensating force. (B,C,D, E) With compensating force for different spread (B) $h = 0.01$, (C) $h=0.04$, (D) $h=0.1$. Solid: cell cortex, dashed: circle with the same circumference. \label{NoObst}}
\end{figure}

\subsection{Simulations of Migration in Channels}

In order to reproduce the biological experiments, our protocol for simulations with channels is started by putting the cell with an initially circular shape close to the opening
of a channel. Then we let it evolve to an equilibrium shape as in the preceding section, after which we turn on a force pushing it into the channel. This is achieved by turning off the pressure force in a region around the
trailing end. The pushing is removed as soon as the cell is entirely inside the channel.

\begin{figure}[H]
\centering
\includegraphics[width=0.9\textwidth]{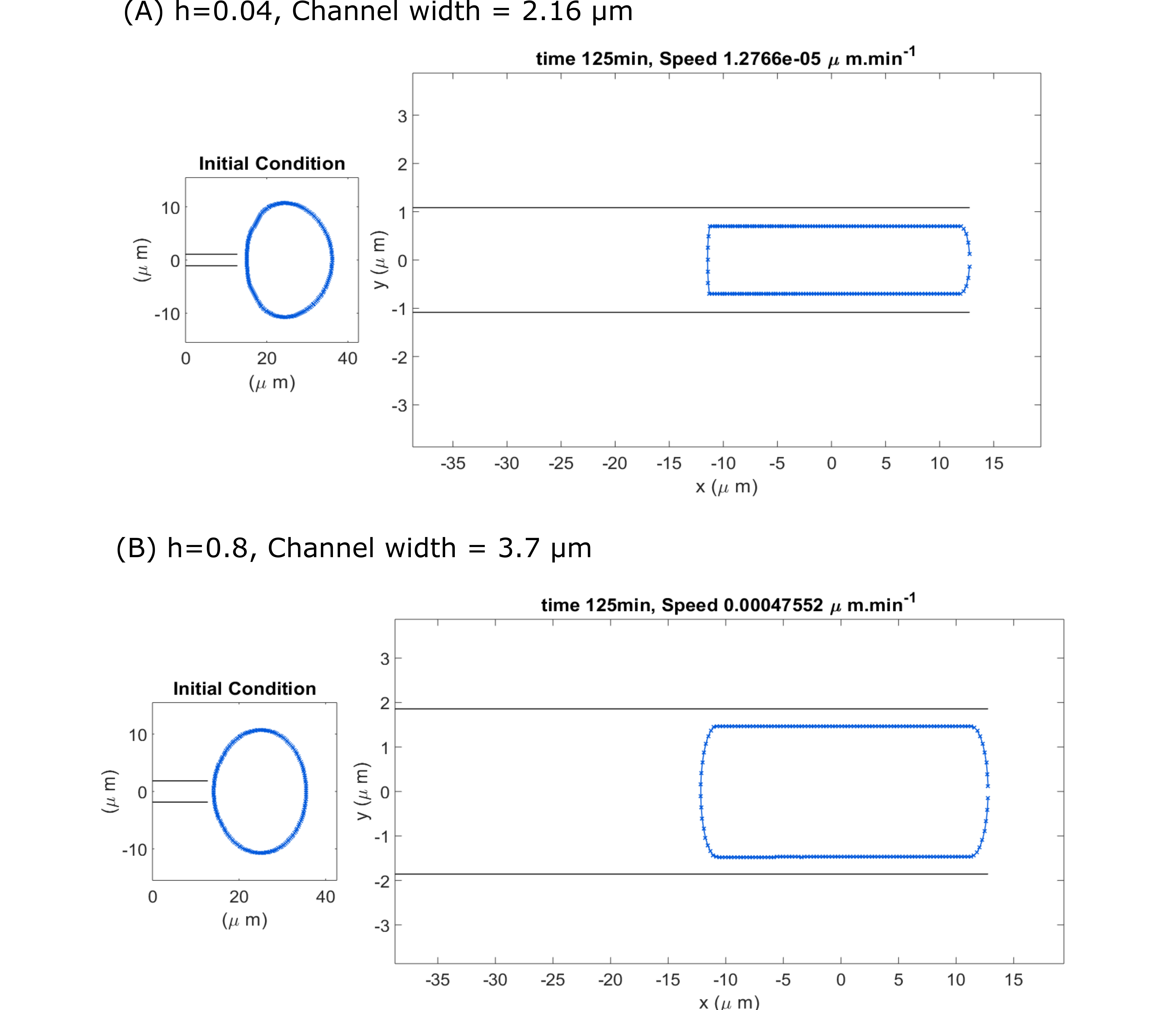}
\caption{Numerical simulations of cells pushed into channels with flat walls for $h=0.04$ and channel width $2.16\mu m$ (A) and for $h = 0.8$ and channel width $3.7 \mu m$ (B). For each simulation, we show on the left side the initial condition, and on the right side the final position. \label{FlatObst}}
\end{figure}

In channels with flat walls with varying widths and spreads of the compensating force (parameter $h$), no migration is observed in
the simulations in spite of polarization and corresponding cortex flow (Figure~\ref{FlatObst}). These results are in
agreement with the experiments described in Section \ref{sec:biological observations}.

Ratchet channels are described by four parameters: (i) a wave length $L_0$ of the width variations, (ii) a minimal width
$2 w_0$, (iii) an amplitude of the width variations $d_0$, and (iv) an asymmetry parameter $\alpha$.
For the length $x$ along the channel, the walls of the channel are given by $\pm f(x)$, with the function
$$
f(x) = d_0 (g(x)-1) - w_0 \,,
$$
where $g(x)$ solves the fixed point equation:
$$
   g(x) = \sin \left(\frac{2 \pi x}{L_0} + \alpha g(x)\right) \,.
$$
In all our numerical simulations we choose $\alpha=0.4$. Note that for $\alpha=0$, the walls would be sinusoidal functions
of $x$, i.e. symmetric with respect to the $y$ axis. As in the experiments, in each simulated channel
we combine three wavelengths, $L_0= 3.9, 7.6, 11.7 \mu m$, increasing in the direction of migration.
Figures~\ref{ratchets} (A,B) show two examples with $2w_0 = 1.36 \mu m$, $d_0 = 0.76 \mu m$ (A) and $2 w_0 = 3.7 \mu m$, $d_0 = 3 \mu m$ (B). As further illustration, Figure~\ref{ratchets} (C) shows a cell at equilibrium before being pushed into a ratchet with $L_0=3.3 \mu m$ and $7.6 \mu m$, $2w_0= 3.7\mu m$, and $d_0=1.54 \mu m$.

\begin{figure}[H]
\centering
\includegraphics[width=\textwidth]{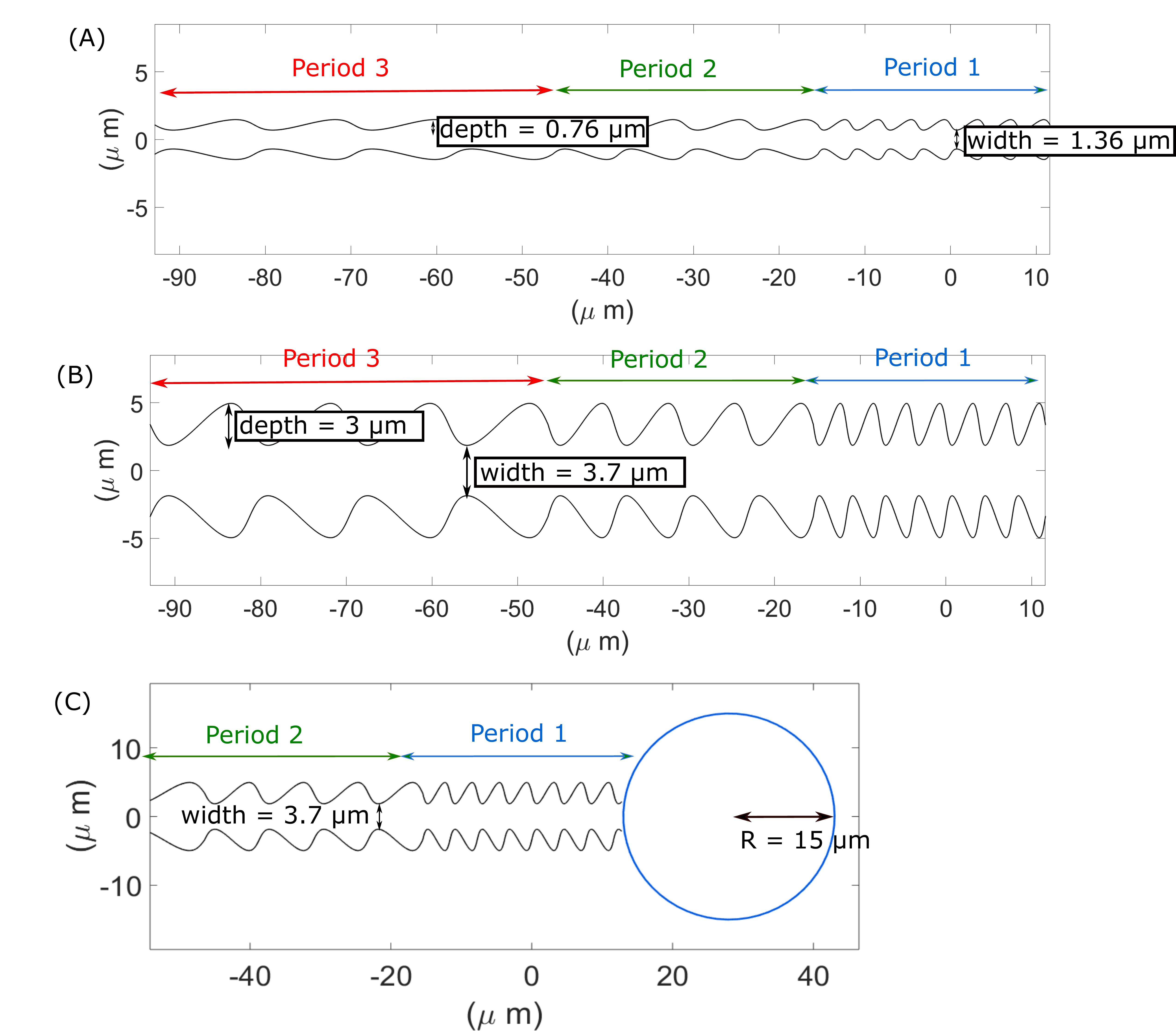}
\caption{Examples of the different obstacle geometries used in the simulations: we explore narrow obstacles of minimal width $1.36 \mu m$ and ratchets of amplitude $0.76 \mu m$ (A) or wide obstacles of width $3.7 \mu m$ and ratchets of amplitude $3 \mu m$ (B). (C) Example of a stationary cell (at equilibrium), before being pushed into a ratchet channel of wavelength $3.3 \mu m$ (Period 1) and $7.6 \mu m$ (Period 2), width $3.7 \mu m$ and amplitude $1.54 \mu m$.\label{ratchets}}
\end{figure}

Results of a typical simulation are shown in Figure~\ref{RatchetObst}. In this situation the cell is able to migrate in all three
different wavelengths. It seems that the average speed is highest in the part with the intermediate wavelength. Within each
wavelength region, the behavior seems to be periodic related to the periodicity of the channel walls. The bigger the
wavelength, the more pronounced are the peaks in the cell speed.

\begin{figure}[H]
\centering
\includegraphics[width=\textwidth]{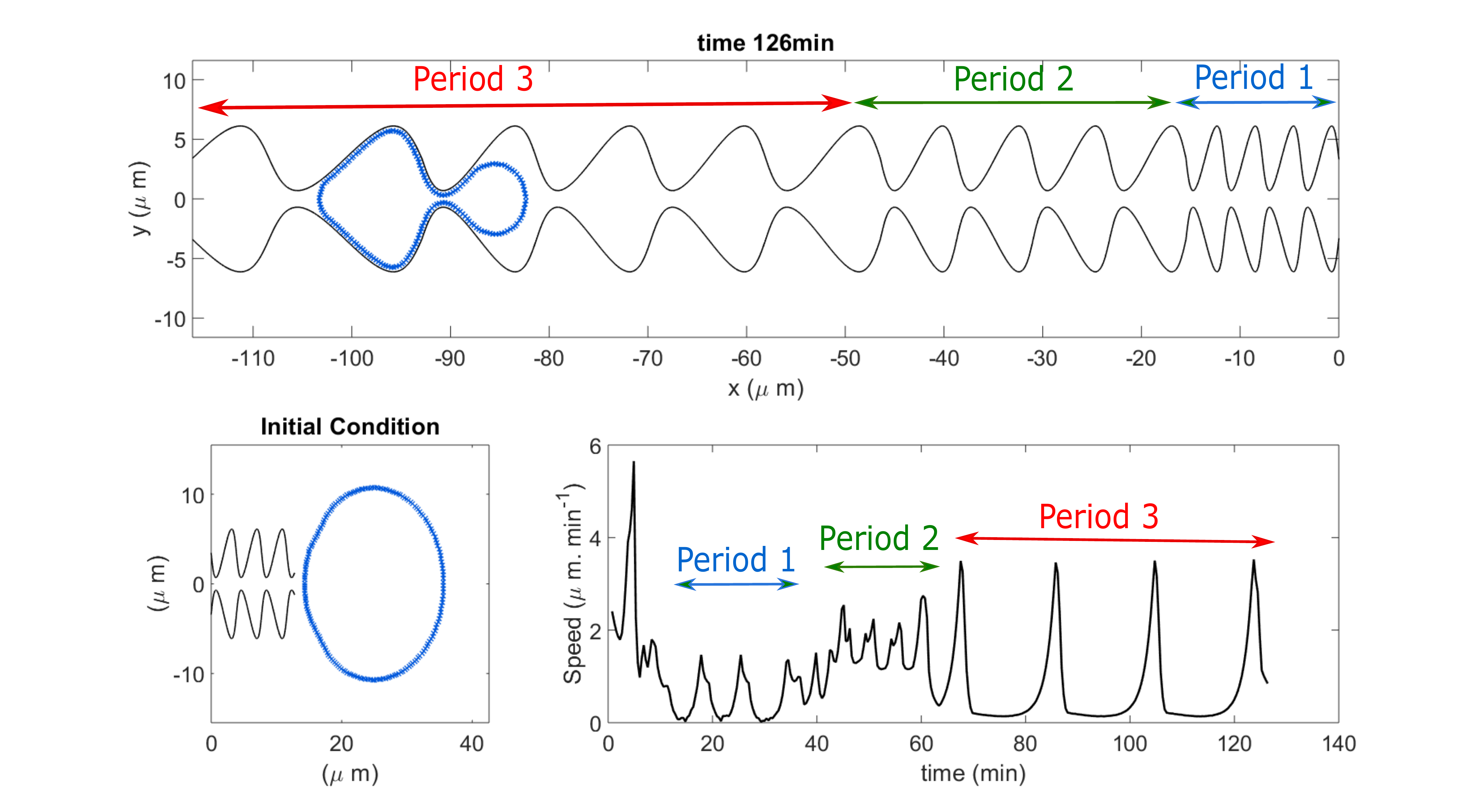}
\caption{Numerical simulation of a cell in a ratchet channel of minimal width 1.4 $\mu$ m, amplitude 2.7 $\mu$ m, and variable wave lengths (3.9, 7.6, and 11.7 $\mu$ m), for $h=0.1$. Bottom left: initial condition. Top: at time $2h$. Bottom right: Speed of the center of gravity vs. time.\label{RatchetObst}}
\end{figure}

A study of the dependence of the mean cell speed on the various parameters has been carried out (Figure~\ref{DiagramSpeed}). Not surprisingly, increasing the channel width decreases the cell speed and can prevent cell migration. For ratchets of the smallest wavelength 3.9 $\mu$ m, the cell is unable to migrate for large amplitudes. In this situation
the cell is unable to enter the ratchets, thus reducing its contact to the wall. The channel acts in these cases as if it had
flat walls. For period $7.6 \mu m$ ratchets, the optimal speed is obtained for large enough amplitude. This is expected, since for small amplitude the ratchet again acts as a flat walled channel. This fact is reinforced in channels of period 11.6 $\mu m$, where only
ratchets with large amplitude can induce cell migration.

To sum up, cell speed is directly linked to the cell compression induced by the channel geometry and to the possibility
for the cell to push against parts of the channel wall facing towards the migration direction.

\begin{figure}[H]
\centering
\includegraphics[width=\textwidth]{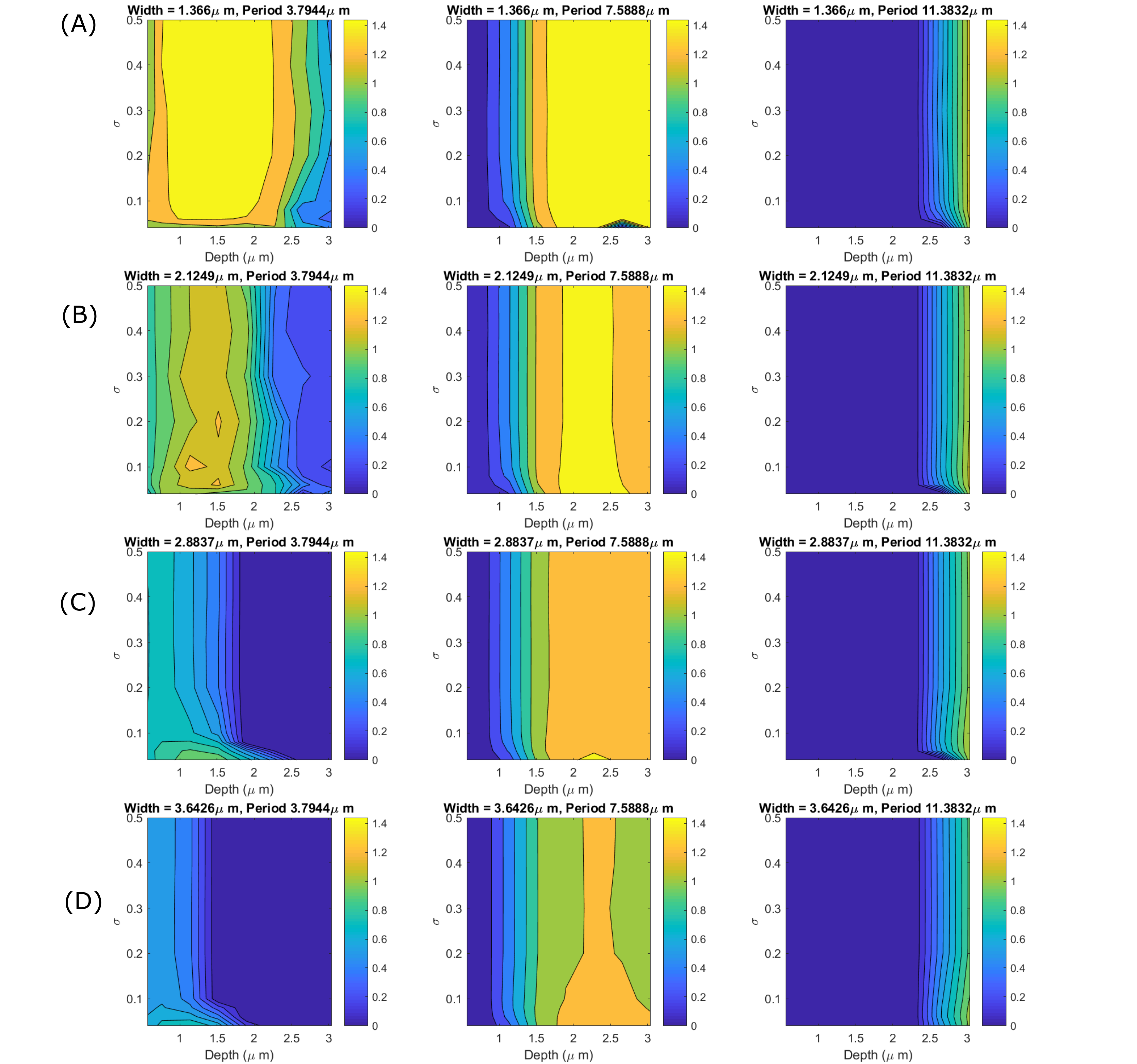}
\caption{Average cell speeds ($\mu m.min^{-1}$) in channels of different minimal widths $w_0$: $1.4 \mu m$ (A), $2.2 \mu m$ (B), $2.9 \mu m$ (C) and $3.7 \mu m$ (D). In each case, the cell speed is represented in dependence of the spread $h\in [0.04,0.8]$ of the
compensating force (vertical direction) and of the amplitude $d_0\in [0.78,3.1] \mu m$ (horizontal direction). We consider channel
walls with wave lengths $3.8 \mu m$ (left figures), $7.6 \mu m$ (middle figures) and $11.7 \mu m$ (right figures). \label{DiagramSpeed}}
\end{figure}

\section{Conclusions}

This study has been motivated by the experimental results described in Section \ref{sec:biological observations},
where adhesion-free migration of leukocytes in artificial, structured micro-channels has been observed.

A mathematical model for this process has been
formulated. Migration is assumed to be due to cortex flow, driven by a local imbalance of polymerization and
depolymerization in a polarized cell. Cell shape is stabilized by cytoplasmic pressure and elastic behavior of the
cortex. An undesired tangential friction, caused by the cortical flow, is balanced by a compensating force, which can
be attributed to the internal transport of depolymerized actin.

The model has the form of an obstacle problem for a strongly nonlinear degenerate parabolic system. Global existence
of solutions has been proven under natural assumptions on the data. The analysis relies on ideas from the theory of
gradient flows, employing the structure of the dominating elastic and pressure terms. The results are complete for an
approximate system with ``softened'' obstacles. The limit for hard obstacles can be carried out, however with an incomplete
characterization of the limiting problem.

For numerical simulations, a conservative, explicit-in-time discretization has been introduced. Under appropriate time step
restrictions, simulations are stable and (at least qualitatively) reproduce the behavior observed in the experiments.
In particular, migration needs both confinement and sufficiently structured channel walls. A parametric study shows the expected dependencies on geometric properties of the channel.

From a modelling point of view, this study has to be seen as a first step. Reliable experimental information on cell cortex
structure and dynamics is still scarce. The fact that in our model migration strongly depends on the force compensating
excess polymerization and depolymerization, is rather questionable. In ongoing work, the model is extended by a
viscous resistance against cortex bending. This effect seems to be a reasonable alternative providing the
necessary pushing force against the channel walls. However, its inclusion poses new challenges both from
an analytic and from a numerical point of view.

%\begin{itemize}
%\item motion in other direction
%\item check biological hypothesis
%\end{itemize}

\appendix
\section{Appendix: A modified Morrey inequality}

\begin{lemma}
Let $u\in L^\infty((0,T); H^1(0,1))\cap H^1((0,T);L^2(0,1)) =: L_t^\infty H_s^1 \cap H_t^1 L_s^2$. Then for any  $(s_0,t_0),(s_1,t_1)\in (0,1)\times (0,T)$ we have:
\begin{equation*}
  |u(s_1,t_1) - u(s_0,t_0)| \le 8\left(  \|\partial_s u\|_{L_t^\infty L_s^2} +
    \|\partial_t u\|_{L_{s,t}^2}\right)(|s_1-s_0|^{1/2} + |t_1-t_0|^{1/4})\,.
\end{equation*}
\end{lemma}

\begin{proof}
We introduce $\Delta s := |s_1-s_0|$ and $\Delta t := |t_1-t_0|$ and consider a rectangle $W\subset (0,1)\times(0,T)$, containing the points $(s_0,t_0),(s_1,t_1)$ with sides parallel to the $t$-axis of lengths $\Delta t$ and parallel to the
$s$-axis of lengths $\Delta s + \sqrt{\Delta t}$, such that $|W| = \Delta t(\Delta s + \sqrt{\Delta t})$. We have
\begin{multline}\label{Delta-u}
  u(s_1,t_1) - u(s_0,t_0) =
  \\\frac{1}{|W|} \int_W (u(s_1,t_1) - u(\sigma,\tau))d(\sigma,\tau) +
  \frac{1}{|W|} \int_W (u(\sigma,\tau) - u(s_0,t_0))d(\sigma,\tau) \,.
\end{multline}
For estimating the second term, we introduce the curve $\{(s_0 + (\sigma-s_0) \sqrt{p},t_0+(\tau-t_0) p):\, 0\le p\le 1\}$,
whence it can be estimated by
\begin{equation*}
   \frac{1}{|W|} \int_0^1 \frac{1}{2\sqrt{p}} \int_W |\sigma-s_0|\, |\partial_s u| d(\sigma,\tau)\, dp +
   \frac{1}{|W|} \int_0^1 \int_W |\tau-t_0|\, |\partial_t u| d(\sigma,\tau)\, dp \,,
\end{equation*}
where the derivatives of $u$ are evaluated along the curve. Employing the Cauchy-Schwarz inequality, this can be
estimated further by
\begin{equation*}
   \frac{\Delta s + \sqrt{\Delta t}}{\sqrt{|W|}} \int_0^1 \frac{1}{2\sqrt{p}} \sqrt{\int_W  |\partial_s u|^2 d(\sigma,\tau)}\, dp +
   \frac{\Delta t}{\sqrt{|W|}} \int_0^1 \sqrt{ \int_W |\partial_t u|^2 d(\sigma,\tau)}\, dp \,.
\end{equation*}
In the second integral over $W$ we introduce the new coordinates $(s,t) = (s_0 + (\sigma-s_0) \sqrt{p},t_0+(\tau-t_0) p)$;
in the first one we first estimate the integrand by its supremum with respect to $t$ and then make the coordinate
transformation only in $s$:
\begin{eqnarray*}
   && \sqrt{\Delta s + \sqrt{\Delta t}} \int_0^1 \frac{dp}{2p^{3/4}} \, \|\partial_s u\|_{L_t^\infty L_s^2} +
   \sqrt{\frac{\Delta t}{\Delta s + \sqrt{\Delta t}}} \int_0^1 \frac{dp}{p^{3/4}} \, \|\partial_t u\|_{L_{s,t}^2} \\
   && \le \left(2  \|\partial_s u\|_{L_t^\infty L_s^2} +
   4 \, \|\partial_t u\|_{L_{s,t}^2}\right) (\Delta s^{1/2} + \Delta t^{1/4})\,.
\end{eqnarray*}
An analogous treatment of the first term on the right hand side of \eqref{Delta-u} completes the proof.
\end{proof}

\paragraph{Acknowledgments}
This work has been supported by the Vienna Science and Technology Fund, Grant no. LS13-029.
G.J. and C.S. also acknowledge support by the Austrian Science Fund, Grants no. W1245, F 65, and W1261, as well as by the Fondation Sciences Math\'ematiques de Paris, and by Paris-Sciences-et-Lettres.

\bibliography{JPRSS.2019}
\bibliographystyle{abbrv}

\end{document}

%% file: imgs/forces.pdf_tex
%% Accompanies image file 'forces.pdf' (pdf, eps, ps)
\begingroup%
  \makeatletter%
  \providecommand\color[2][]{%
    \errmessage{(Inkscape) Color is used for the text in Inkscape, but the package 'color.sty' is not loaded}%
    \renewcommand\color[2][]{}%
  }%
  \providecommand\transparent[1]{%
    \errmessage{(Inkscape) Transparency is used (non-zero) for the text in Inkscape, but the package 'transparent.sty' is not loaded}%
    \renewcommand\transparent[1]{}%
  }%
  \providecommand\rotatebox[2]{#2}%
  \newcommand*\fsize{\dimexpr\f@size pt\relax}%
  \newcommand*\lineheight[1]{\fontsize{\fsize}{#1\fsize}\selectfont}%
  \ifx\svgwidth\undefined%
    \setlength{\unitlength}{371.77150456bp}%
    \ifx\svgscale\undefined%
      \relax%
    \else%
      \setlength{\unitlength}{\unitlength * \real{\svgscale}}%
    \fi%
  \else%
    \setlength{\unitlength}{\svgwidth}%
  \fi%
  \global\let\svgwidth\undefined%
  \global\let\svgscale\undefined%
  \makeatother%
  \begin{picture}(1,0.48094794)%
    \lineheight{1}%
    \setlength\tabcolsep{0pt}%
    \put(0,0){\includegraphics[width=\unitlength,page=1]{imgs/forces.pdf}}%
    \put(0.27436207,0.27491369){\color[rgb]{0,0,0}\makebox(0,0)[lt]{\begin{minipage}{0.42768206\unitlength}\raggedright \end{minipage}}}%
    \put(0.17551103,0.5220413){\color[rgb]{0,0,0}\makebox(0,0)[lt]{\begin{minipage}{0.20035123\unitlength}\raggedright \end{minipage}}}%
    \put(0.72751229,0.10770038){\color[rgb]{0,0,0}\makebox(0,0)[lt]{\lineheight{1.25}\smash{\begin{tabular}[t]{l}\circled{2}\end{tabular}}}}%
    \put(0.05906413,0.41921033){\color[rgb]{0,0,0}\makebox(0,0)[lt]{\lineheight{1.25}\smash{\begin{tabular}[t]{l}\circled{1}\end{tabular}}}}%
    \put(0.67887612,0.42562955){\color[rgb]{0,0,0}\makebox(0,0)[lt]{\lineheight{1.25}\smash{\begin{tabular}[t]{l}\circled{5}\end{tabular}}}}%
    \put(0.92494638,0.35287837){\color[rgb]{0,0,0}\makebox(0,0)[lt]{\lineheight{1.25}\smash{\begin{tabular}[t]{l}\circled{4}\end{tabular}}}}%
    \put(0.1476285,0.21455802){\color[rgb]{1,1,1}\rotatebox{-8.05347749}{\makebox(0,0)[lt]{\begin{minipage}{0.60942188\unitlength}\raggedright \textcolor{white}{actin transport}\end{minipage}}}}%
    \put(0.01484279,0.2815536){\color[rgb]{0,0,0}\makebox(0,0)[lt]{\lineheight{1.25}\smash{\begin{tabular}[t]{l}\circled{3}\end{tabular}}}}%
  \end{picture}%
\endgroup%

%% file: imgs/cell_parametrization.pdf_tex
%% Accompanies image file 'cell_parametrization.pdf' (pdf, eps, ps)
\begingroup%
  \makeatletter%
  \providecommand\color[2][]{%
    \errmessage{(Inkscape) Color is used for the text in Inkscape, but the package 'color.sty' is not loaded}%
    \renewcommand\color[2][]{}%
  }%
  \providecommand\transparent[1]{%
    \errmessage{(Inkscape) Transparency is used (non-zero) for the text in Inkscape, but the package 'transparent.sty' is not loaded}%
    \renewcommand\transparent[1]{}%
  }%
  \providecommand\rotatebox[2]{#2}%
  \newcommand*\fsize{\dimexpr\f@size pt\relax}%
  \newcommand*\lineheight[1]{\fontsize{\fsize}{#1\fsize}\selectfont}%
  \ifx\svgwidth\undefined%
    \setlength{\unitlength}{349.58102489bp}%
    \ifx\svgscale\undefined%
      \relax%
    \else%
      \setlength{\unitlength}{\unitlength * \real{\svgscale}}%
    \fi%
  \else%
    \setlength{\unitlength}{\svgwidth}%
  \fi%
  \global\let\svgwidth\undefined%
  \global\let\svgscale\undefined%
  \makeatother%
  \begin{picture}(1,0.46454352)%
    \lineheight{1}%
    \setlength\tabcolsep{0pt}%
    \put(0,0){\includegraphics[width=\unitlength,page=1]{imgs/cell_parametrization.pdf}}%
    \put(0.55406796,0.41671506){\color[rgb]{0,0,0}\makebox(0,0)[lt]{\lineheight{0}\smash{\begin{tabular}[t]{l}$\tau = \frac{\partial_s X}{|\partial_s X|}$\end{tabular}}}}%
    \put(0.94069847,0.32517119){\color[rgb]{0,0,0}\makebox(0,0)[lt]{\lineheight{0}\smash{\begin{tabular}[t]{l}$n=-\tau^\perp$\end{tabular}}}}%
    \put(0,0){\includegraphics[width=\unitlength,page=2]{imgs/cell_parametrization.pdf}}%
    \put(0.6240214,0.17445163){\color[rgb]{0,0,0}\makebox(0,0)[lt]{\lineheight{0}\smash{\begin{tabular}[t]{l}$\partial_s X^\perp$\end{tabular}}}}%
  \end{picture}%
\endgroup%